\documentclass[11pt]{amsart}
\usepackage{amsfonts,amssymb,amsmath,amsthm,here}

\usepackage{color}  
\usepackage{mathrsfs}

\theoremstyle{plain}
\newtheorem{theo}{Theorem}

\newtheorem{lem}{Lemma}

\theoremstyle{definition}

\newtheorem{ex}{Example}

\newtheorem{rem}{Remark}

\makeatletter

\@addtoreset{equation}{section}
\makeatother

\makeatletter
\@namedef{subjclassname@2020}{
  \textup{2020} Mathematics Subject Classification}
\makeatother

\def \Z{{\mathbb Z}}
\def \Q{{\mathbb Q}}

\def \C{{\mathbb C}}
\def \calO{{\mathcal O}}
\def \calA{{\mathcal A}}

\def \calD{{\mathcal D}}

\def \Gal{\mathrm{Gal}}
\def \Tr{\mathrm{Tr}}

\allowdisplaybreaks[4]

\author{Hajime Ogawa}
\address{Department of Mathematics,
Interdisciplinary Faculty of Science and Engineering,
Shimane University,
Matsue, Shimane, 690-8504, Japan}
\email{hollyredp2977@gmail.com}

\author{Miho Aoki }
\address{Department of Mathematics,
Interdisciplinary Faculty of Science and Engineering,
Shimane University,
Matsue, Shimane, 690-8504, Japan}
\email{aoki@riko.shimane-u.ac.jp}

\subjclass[2020]{Primary 11R04, 11R16,  Secondary 11C08, 11R80. }
\keywords{the ring of integers, simplest cubic fields, Shanks' cubic polynomial}
\thanks{This work was supported by JSPS KAKENHI Grant Number JP21K03181} 

\title[Simplest cubic fields]{Galois module structure of algebraic integers \\ of the simplest cubic field}

\begin{document}

\begin{abstract} 
Let $L_n$ be a simplest cubic field with  Galois group $G=\Gal (L_n/\Q)$.
The associated order is denoted as $\calA_{L_n/\Q}:= \{ x\in \Q [G] \, |\, x \cdot \calO_{L_n} \subset \calO_{L_n } \}$,
where $\calO_{L_n}$ is the ring of integers of $L_n$.
Leopoldt showed that $\calO_{L_n} \simeq \calA_{L_n/\Q}$ as $\calA_{L_n/\Q}$-modules.
In this paper, 
we give a generator of the $\calA_{L_n/\Q}$-module $\calO_{L_n}$ explicitly using the roots of Shanks' cubic polynomial.
If $L_n/\Q$ is tamely ramified, then we have $\calA_{L_n/\Q}=\Z [G]$, 
and the conjugates form a normal 
integral basis, which has been obtained explicitly in the previous work of  Hashimoto and the second author.
\end{abstract}
\maketitle
\section{Introduction}\label{sec:intro}

Let $n$ be an integer.
We consider the Shanks' cubic polynomial $f_n(X)\in\Z[X]$ defined by
\begin{equation}\notag
f_n(X)=X^3-nX^2-(n+3)X-1.
\end{equation}
The polynomial $f_n(X)$ is irreducible for all $n\in\Z$,
and $L_n:=\Q(\rho_n)$ is a cyclic cubic field, where $\rho_n$ is a root of 
$f_n(X)$. Let $\sigma$ be the generator of $G:=$Gal$(L_n/\Q)$ that satisfies
$\sigma (\rho_n)=-1/(1+\rho_n)$.
For any $x\in L_n$, we put 
$x':=\sigma (x)$ and $x'':=\sigma^2 (x)$.
The field is known as   {\it  a simplest cubic field}. 
See \cite{S} for details.
We have  $L_n=L_{-n-3}$ since $f_n(X)=-X^3f_{-n-3}(1/X)$.
We express $\Delta_n:=n^2+3n+9=bc^3$ 
using  $b,c\in\Z_{>0}$, where $b$ is cube-free.
The discriminant of $f_n(X)$ is $d(f_n)=\Delta_n^2$, 
and the conductor $\mathfrak{f}_{L_n}$ 
of $L_n$ can be expressed as follows
(see Cusick \cite[Lemma 1]{C} when that $\Delta_n$ is square-free,
Washington \cite[Proposition 1]{W2} when $n\not\equiv3 \pmod{9}$, Komatsu \cite[Lemma~1.4]{K} and Kashio-Sekigawa \cite[Remark~1.5]{KS} in general):
\begin{equation}\label{eq:cond}
\mathfrak{f}_{L_n}=\gamma\prod_{\substack{
p\mid b
\\
p\neq3
}}p,\ \ 
\gamma=
\begin{cases} 
1, & \text{if}\ 3\nmid n\ \text{or}\ n\equiv12 \pmod{27},
\\
3^2, & \text{otherwise},
\end{cases}
\end{equation}
where the product runs over all prime numbers $p$ dividing $b$.
The discriminant of $L_n$ is $D_{L_n}=\mathfrak{f}_{L_n}^2$ from the conductor-discriminant formula 
 (\cite[Theorem~3.11]{W1}).
Let $\calO_{L_n}$ be  the ring of integers of $L_n$.
Leopoldt showed that $\calO_{L_n} \simeq \calA_{L_n/\Q}$ as $\calA_{L_n/\Q}$-modules,
where  $\calA_{L_n/\Q}:= \{ x\in \Q [G] \, |\, x \cdot \calO_{L_n} \subset \calO_{L_n } \}$
is the associated order for $L_n/\Q$. 
In this paper, 
we give a generator of the $\calA_{L_n/\Q}$-module $\calO_{L_n}$ explicitly using the roots of the polynomial $f_n(X)$.
If $L_n/\Q$ is tamely ramified, then it  has been obtained in the previous work of  Hashimoto and the second author
\cite{HA1}, which is a generalization of works of
 Lehmer \cite{Leh} (when $\Delta_n =\mathfrak f_{L_n}$ is a prime number),  Ch\^{a}telet \cite{Ch} and Lazarus \cite{La}  (when  $\Delta_n =\mathfrak f_{L_n}$ is 
square-free).
\section{Branch classes and associated orders}\label{sec:BA}

This section discusses the Galois module structure of the ring of integers of an abelian number field
that was obtained by Leopoldt \cite{Leo}.
We have a well-written article \cite{Le} by Lettl that simplifies Leopoldt's work.
Let $L$ be an abelian number field with  Galois group $G$ and conductor $\mathfrak f$. Let $\mathscr X$  be the  group of Dirichlet characters associated to $L$.
Let $v_p(x)$  denote the $p$-adic valuation of $x \in \Q$ for a prime number $p$.
For any $m\in \mathbb Z_{>0}$, we put
\[
p (m) :=\prod_{\substack{ 
p\mid m
\\
p\neq 2
}} p,\qquad 
q(m) := \prod_{\substack{
p \\
v_p(m) \geq 2}} p^{v_p(m)},
\]
where the first product runs over all odd prime numbers $p$ dividing $m$, and the second product runs over all
prime numbers $p$ that satisfy $v_p (m)\geq 2$.
Put
\[
\calD (\mathfrak f):= \{ m\in \Z_{>0} \, |\, p(\mathfrak f) | m,\ m| \mathfrak f,\ m\not\equiv 2 \pmod{4}\}.
\]
 We define  {\it a branch class} of $\mathscr X$ for any $m\in \calD (\mathfrak f)$ as follows:
\[
\Phi_m := \{ \chi \in \mathscr{X}  \, |\, q(\mathfrak f_{\chi})=q(m) \},
\]
where $\mathfrak f_{\chi}$ is the conductor of $\chi$.
We have $\mathscr{X}=\coprod_{m\in \calD (\mathfrak f)} \Phi_m$ (disjoint union). For any $\chi \in X$,
let
\[
e_{\chi} := \frac{1}{[L:\Q]} \sum_{g \in G} \chi^{-1} (g) g
\]
be the idempotent. Furthermore, for any $m \in \calD (\mathfrak f)$, let
\[
e_m := \sum_{\chi \in \Phi_m} e_{\chi}.
\]
Since the branch class $\Phi_m$ is closed under conjugation,
we obtain $e_m \in \Q [G]$.
Let $\calO_L$ be the ring of integers of $L$.
The action of the group ring $\Q [G]$ on $L$ is given as follows:
$
x \cdot a :=\sum_{g\in G} n_{g} g (a) \quad \left( x=\sum_{g \in G} n_{g} g \in \Q  [G],\, a\in L \right).
$
We define  {\it the associated order} for $L/\Q$ as follows:
\[
\calA_{L/\Q}:= \{ x\in \Q [G] \, |\, x\cdot \calO_{L} \subset \calO_{L} \}.
\]
Leopoldt showed that $\calO_{L}$ is a free module of rank $1$ over $\calA_{L/\Q}$.
The generator $T \, (\in \calO_L)$ is given by
\[
T:=\sum_{m \in \calD (\mathfrak f)} \eta_m,\quad \eta_m:= \Tr_{\Q (\zeta_m)/(L\cap \Q(\zeta_m))}( \zeta_m).
\]
\begin{theo}[Leopoldt \cite{Leo}]\label{theo:Leopoldt}
We have
\[
\calO_L=\bigoplus_{m \in \calD (\mathfrak f)} \Z [G]\cdot \eta_m =\calA_{L/\Q}\cdot T
\]
and
\[
\calA_{L/\Q}=\Z [G][\{e_m\, |\, m\in \calD (\mathfrak f) \}].
\]
\end{theo}

Let $L_n$ be the simplest cubic field defined in \S \ref{sec:intro}.
We express $\Delta_n=n^2+3n+9=de^2c^3$ using $d,e,c\in\mathbb{Z}_{>0}$, where
$d$ and $e$ are square-free and $(d,e)=1$.
According to (\ref{eq:cond}), the conductor of $L_n$ is given as follows.

\begin{enumerate}
\item[(T1)]\  If $3 \nmid n$, then we have $3\nmid d, 3\nmid e, 3\nmid c$ and
$\mathfrak f_{L_n}=de$.
\item[(T2)]\ If $n\equiv 12 \pmod{27}$, then we have $3\nmid d, 3\nmid e, 3 ||  c$ and
$\mathfrak f_{L_n}=de$.
\item[(W1)]\ If $n\equiv 3\pmod{9}, n\not\equiv 12\pmod{27}$, then we have $3\nmid d, 3\nmid e, 3 || c$ and
$\mathfrak f_{L_n}=9de$.
\item[(W2)]\ If $n\equiv 0,6 \pmod{9}$, then we have $3\nmid d, 3 || e, 3\nmid c$ and
$\mathfrak f_{L_n}=3de$.
\end{enumerate}

The set $\calD (\mathfrak f_{L_n})$ is
\begin{equation}\notag
\calD (\mathfrak f_{L_n }) =\begin{cases}
\{ de \} & \text{(T1 or T2)}, \\
\{ 9de, 3de \} &  \text{(W1)}, \\
\{ 3de, de \} &  \text{(W2)},
\end{cases}
\end{equation}
and hence the associated order $\calA_{L_n/\Q}$ and the generator $T$ of $\calO_{L_n}$ are given by
\begin{equation}\notag
\calA_{L_n/\Q}=\begin{cases}
\Z [G] [e_{de}] &  \text{(T1 or T2)}, \\
\Z [G] [e_{9de},e_{3de}] &   \text{(W1)}, \\
\Z [G][e_{3de}, e_{de} ] &  \text{(W2)},
\end{cases}
\end{equation}
\begin{equation}\notag
T= \sum_{m \in \calD (\mathfrak f_{L_n} )} \eta_m =\begin{cases}
\eta_{de}& \text{(T1 or T2)}, \\
\eta_{9de}+\eta_{3de} &  \text{(W1)}, \\
\eta_{3de}+\eta_{de}&  \text{(W2)}.
\end{cases}
\end{equation}
\section{Tamely ramified cases}\label{sec:tame}

Suppose $3 \nmid n$ or $n\equiv 12 \pmod{27}$. Since the conductor of $L_n$ is 
$\mathfrak f_{L_n}=de$, which is  square-free, $L_n/\Q$ is tamely ramified.
In this case, the corresponding group of Dirichlet characters is $\mathscr{X}=\Phi_{de} =
\{ {\bf 1}, \chi, \chi^2 \}$, where $\bf 1$ is the trivial character and $\chi : G \to \C^{\times},\
\sigma \mapsto \zeta_3$. Since $e_{de}=1$, we have $\calA_{L_n/\Q}=\Z [G][e_{de}]=\Z [G]$ and
$\calO_{L_n}=\Z [G]\cdot  \eta_{de}$.
The ring of integers $\calO_{L_n}$ is isomorphic to $\Z [G]$ as $\Z [G]$-module, and the conjugates $\{
\eta_{de}, \eta_{de}', \eta_{de}'' \}$ of $\eta_{de}$ is a normal integral basis of $L_n$.
 Put $\rho := \rho_n$, $\zeta := \zeta_3$ and  $\Delta_n =n^2+3n+9=A_n A_n'$ with
$A_n:=n+3(1+\zeta),\, A_n':=n+3(1+\zeta^2)$.
In this case, the generator of the $\Z [G]$-module $\calO_{L_n}$ has been obtained in the previous work \cite{HA1}.
\begin{theo}[\cite{HA1}]\label{theo:HA}
Suppose   $3 \nmid n$ or $n\equiv 12 \pmod{27}$. 
Let $a_0$ and $a_1$ be integers that satisfy $ec=a_0^2-a_0a_1+a_1^2$ and
$a_0+a_1\zeta\mid A_n$ in $\Z [\zeta]$.
Put $m=(\varepsilon ec^2-n(a_0+a_1))/3\in\mathbb{Z}$, where
$\varepsilon\in\{\pm1\}$ is given by
\[
\varepsilon: =
\begin{cases}
\left( \frac{n (a_0+a_1)}{3} \right) & \rm{(T1)}, 
\\
\left( \frac{a_0}{3} \right) & \rm{(T2)},
\end{cases}
\]
where $\left( \frac{\cdot}{3} \right)$ is the Legendre symbol.
Then
\begin{equation*}
\alpha :=\frac{1}{ec^2}(a_0\rho_n+a_1\rho_n'+m)
\end{equation*}
is a generator of a normal integral basis of the simplest cubic field $L_n$, namely we have
$\calO_{L_n}=\Z [G] \cdot \alpha$.
\end{theo}

Since $\Z [G]^{\times} =\{ \pm 1_G, \pm \sigma, \pm \sigma^2\}$,
a normal integral basis 
$ \{ \eta_{de}, \eta_{de}' ,\eta_{de}'' \}$ 
coincides with $ \{ \alpha, \alpha', \alpha'' \}$ or $\{ -\alpha,
-\alpha', -\alpha'' \}$ where $\alpha$ is the element defined in the theorem above.
\begin{rem}\label{rem:quintic}
A  result similar to that of  Theorem~\ref{theo:HA} was recently obtained for Lehmer's cyclic quintic fields \cite{HA2}.
\end{rem}

\section{Wildly ramified cases}\label{sec:wild}

Suppose   $3 | n$ and $n\not\equiv 12 \pmod{27}$. Since the conductor of $L_n$ is 
$\mathfrak f_{L_n}=9de$ if $n\equiv 3 \pmod{9}, n\not\equiv 12 \pmod{27}$ and $3de$ if $n\equiv 0,6 \pmod{9}$,
$L_n/\Q$ is wildly ramified. Put $\rho := \rho_n$, $\zeta := \zeta_3$ and  $\Delta_n =n^2+3n+9=A_n A_n'$ with
$A_n:=n+3(1+\zeta),\, A_n':=n+3(1+\zeta^2)$.

\begin{theo}\label{theo:OA}
Suppose   $3 | n$ and $n\not\equiv 12 \pmod{27}$.
Let $a_0$ and $a_1$ be integers that satisfy the following:
\[ec=
\begin{cases}
a_0^2-a_0a_1+a_1^2 &  \rm{(W1)}, \\
3(a_0^2-a_0a_1+a_1^2)  &  \rm{(W2)},
\end{cases}
\]
and
$a_0+a_1\zeta\mid A_n$ in $\Z [\zeta]$.
Put 
\[
\alpha :=\frac{1}{ec^2}(3a_0\rho+3a_1\rho'-n(a_0+a_1)).
\]
Then we have
\[
\calO_{L_n}=\Z [G] \cdot \alpha \oplus \Z =\calA_{L_n/\Q} \cdot (\alpha+1).
\]
\end{theo}
For the proof of Theorem~\ref{theo:OA}, we first find an integral basis of $L_n$ 
using the following lemma. For $x_1, x_2, x_3 \in L_n$, we define
the discriminant:
\[
d(x_1,x_2,x_3):= \begin{vmatrix}
x_1 & x_2 & x_3 \\
x_1' & x_2' & x_3'  \\
x_1'' & x_2'' & x_3'' \\
\end{vmatrix}^2.
\]
\begin{lem}\label{lem:HW}
{\cite[p.15, Theorem~1.5]{HW}}
Let $g(X)=X^3+b_1X^2+c_1X+d_1$ be an irreducible polynomial in
$\mathbb{Z}[X]$, and $\rho$ be a root of $g(X)$.
Put $K=\mathbb{Q}(\rho)$.
Suppose  $s,a,t\in\mathbb{Z}$ satisfy the following conditions
$(\mathrm{i})$ and $(\mathrm{ii})$.
\begin{align*}
\mathrm{(i)}\quad & d(1,\rho,\rho^2) =s^6a^2D_K,\\
\mathrm{(ii)}\quad & \dfrac{1}{2}g''(t)
\equiv 
0 \pmod{s},
\\
&g'(t)\equiv 
0 \pmod{s^2a},
\\
&g(t)\equiv
0 \pmod{s^3a^2}, 
\end{align*}
where $g'(X)\ (resp.\ g''(X))$ is the derivative
$(resp.\ the\ second\ derivative)$ of $g(X)$.
Put
\begin{eqnarray*}
\phi
\!\!\!&:=&\!\!\!
\frac{1}{s}(\rho-t),
\\
\psi
\!\!\!&:=&\!\!\!
\frac{1}{s^2a}(\rho^2+(t+b_1)\rho+t^2+b_1t+c_1).
\end{eqnarray*}
Then $\{1,\phi,\psi\}$ is an integral basis of $K$.
\end{lem}
\begin{lem}\label{lem:IB}
Suppose  $3 | n$ and $n\not\equiv 12 \pmod{27}$. Put
\begin{align*}
&\phi:=
\begin{cases}
\frac{1}{c} (3\rho-n) &  \mathrm{if}\ n\equiv 3 \!\!\! \pmod{9}, \ n\not\equiv 12 \!\!\! \pmod{27}, \\
\frac{1}{3c}(3\rho -n) &  \mathrm{if}\ n\equiv 0,6 \!\!\! \pmod{9} ,\\
\end{cases}
\\
&\psi:= 
\frac{1}{3ec^2} (9\rho^2-6n\rho-2n^2-9n-27).
\end{align*}
Then, $\{1, \phi, \psi \}$ is an integral basis of $L_n$.
\end{lem}
\begin{proof}
The discriminant of $L_n$ is
\[
D_{L_n} =\begin{cases}
(9de)^2& \text{(W1)}, \\
(3de)^2&   \text{(W2)}. 
\end{cases}
\]
Let 
\begin{align*}
s &:=\begin{cases}
\frac{c}{3} & \text{(W1)}, \\
c&  \text{(W2)}, 
\end{cases}\\
a & := \begin{cases}
3e& \text{(W1)}, \\
\frac{e}{3}& \text{(W2)}, 
\end{cases}\\
t & :=\frac{n}{3}.
\end{align*}
Then, the integers $s, a$ and $t$ satisfy the conditions (i) and (ii) of Lemma~\ref{lem:HW}, 
and hence we get the assertion.
\end{proof}

\begin{proof}[Proof of Theorem~\ref{theo:OA}]
First, we consider the case of $n\equiv 3 \pmod{9}, \ n\not\equiv 12 \pmod{27}$.
In this case, we have $3\nmid d,\, 3\nmid e,\, 3||c$ and the conductor is 
$\mathfrak f_{L_n}=9de$. The corresponding group of Dirichlet characters is $\mathscr{X}=\Phi_{9de} \coprod \Phi_{3de}$ with
$\Phi_{9de}=\{ \chi, \chi^2 \},\, \Phi_{3de}= \{\bf 1 \}$,
where $\bf 1$ is the trivial character and $\chi : G \to \C^{\times},\
\sigma \mapsto \zeta_3$. We have $e_{9de}=e_{\chi}+e_{\chi^2}= (2-\sigma-\sigma^2)/3,\, 
e_{3de}=e_{\bf 1} =(1+ \sigma+\sigma^2)/3$, $\calD(\mathfrak f_{L_n})=\{ 9de,3de \}$,
$\calA_{L_n/\Q}=\Z [G][e_{9de},e_{3de}]$ and
$\calO_{L_n}=\Z [G] \cdot \eta_{9de} \oplus \Z [G]\cdot \eta_{3de}$.
Since $\eta_{3de}=\Tr_{\Q (\zeta_{3de})/(L_n\cap \Q(\zeta_{3de}))}  (\zeta_{3de})
=\Tr_{\Q (\zeta_{3de})/ \Q} (\zeta_{3de})=\mu (3de) (=\pm 1)$,
where $\mu$ is the M\"{o}bius function, we have the following: 
\begin{equation}\label{eq:O1}
\calO_{L_n} =\Z [G]\cdot \eta_{9de} \oplus \Z [G] \cdot \eta_{3de} = \Z [G] \cdot \eta_{9de} \oplus \Z   =\calA_{L_n/\Q} \cdot (\eta_{9de}+1).
\end{equation}
Let $\alpha$ be an element of $\calO_{L_n}$ that satisfies the following: 
\begin{equation}\label{eq:alpha1}
\Z [G]\cdot \alpha=\Z [G]\cdot \eta_{9de}.
\end{equation}
Since
\begin{equation}\label{eq:em}
e_m\cdot  \eta_{m'} = \begin{cases}
0 & \it{if}\ m\ne m', \\
\eta_m &  \it{if} \ m=m'\\
\end{cases}
\end{equation}
for any $m,m' \in \calD (\mathfrak f_{L_n})$ (see \cite[p.165, (2) and p.167, lines 1,2]{Le}), we have
$e_{9de}\cdot \calO_{L_n} =\Z [G]\cdot \alpha$
from (\ref{eq:O1}) and (\ref{eq:alpha1}). Let $ \{1,\phi,\psi \}$ be the integral basis of $L_n$ given by
Lemma~\ref{lem:IB}. 
We find $\alpha$ using the method proposed by Acciaro and Fieker \cite{AF}.
Since $e_{9de}\cdot  \Z=0$, we have 
\begin{equation}\label{eq:eO1}
e_{9de}\cdot  \calO_{L_n} =\Z [G]\cdot \alpha=(e_{9de}\cdot \phi) \Z +(e_{9de}\cdot \psi)\Z.
\end{equation}
Using $\rho^2=\rho'+(n+1)\rho+2$, we obtain the following:
\begin{align}
e_{9de}\cdot  \phi &=\frac{1}{c} (2\rho-\rho'-\rho''), \label{eq:e-phi-psi1}\\
e_{9de}\cdot \psi & =\frac{1}{3ec^2} ((2n+3)\rho+(3-n)\rho' -(n+6)\rho'') \notag .
\end{align}
Let $\ell :=3ec^2$. From (\ref{eq:eO1}) and (\ref{eq:e-phi-psi1}), there exists $g\in \Z [G]$
satisfying
\begin{equation}\label{eq:alpha-g-1}
\alpha= g \cdot \frac{\rho}{\ell}
\end{equation}
and
\begin{align}
e_{9de}. \phi =g_1\cdot  \frac{\rho}{\ell}, \hspace{10mm} & g_1:= \frac{\ell}{c} (2-\sigma -\sigma^2) \label{eq:e-phi-psi-g1},\\
e_{9de}. \psi =g_2\cdot  \frac{\rho}{\ell}, \hspace{10mm} & g_2:=\frac{\ell}{3ec^2} ((2n+3)+(3-n)\sigma -(n+6)\sigma^2) \notag .
\end{align}
From (\ref{eq:eO1}), (\ref{eq:alpha-g-1}) and (\ref{eq:e-phi-psi-g1}), we obtain the following:
\[
g\Z [G]\cdot  \frac{\rho}{\ell} =(g_1 \Z [G] +g_2 \Z [G] )\cdot  \frac{\rho}{\ell},
\]
and hence we have the equality of ideals of $\Z [G]$:
\[
(g) +\mathrm{Ann}_{\Z [G] }\left( \frac{\rho}{\ell} \right) =(g_1,g_2) +\mathrm{Ann}_{\Z [G] }\left( 
\frac{\rho}{\ell} \right).
\]
Since $d(\rho, \rho',\rho'') \ne 0$ (\cite[Proof of Theorem~4.4]{HA1}), $\{ \rho/\ell, \rho'/\ell, \rho''/\ell \}$
is a normal basis of $L_n$, we obtain $\mathrm{Ann}_{\Z [G] }(\rho/\ell)=0$,
and hence we obtain an   equality of ideals of  $\Z [G]$:
\[
(g)=(g_1,g_2).
\]
Consider the surjective ring homomorphism
\[
\nu: \Z [G] \longrightarrow \Z [\zeta],
\]
defined by $\nu (\sigma )=\zeta$. We calculate the image of the ideal $I:=(g)=(g_1,g_2)$ of $\Z [G]$.
 Since $\nu$ is surjective, we obtain the ideal of $\Z [\zeta]$:
\begin{equation}\notag
\nu (I) =( \nu (g)) =(\nu (g_1), \nu(g_2)).
\end{equation}
 From (\ref{eq:e-phi-psi-g1}), the elements $\nu (g_1)$ and $\nu (g_2)$ are given by
 \begin{align*}
 \nu (g_1) & =\frac{\ell}{c} (2-\zeta-\zeta^2)=9ec, \\
 \nu (g_2) & = \frac{\ell}{3ec^2} ((2n+3)+(3-n)\zeta -(n+6)\zeta^2) =9 \times \frac{A_n}{3} 
 \end{align*}
 with $A_n/3=n/3+\zeta+1 \in \Z [\zeta]$. Therefore, we have
 \begin{equation}\label{eq:nuI-gen-1}
 \nu (I)=(\nu (g))=9 (ec,A_n/3).
 \end{equation}
 We consider the decompositions of $ec$ and $A_n/3$ into  products of prime elements. 
 First, we consider the prime element $1-\zeta$ of $\Z [\zeta]$ above $3$.
 Since $n\equiv 3 \pmod{9}$, we have
 $A_n/3 =n/3 +1+\zeta \equiv 0 \pmod{(1-\zeta)}$. Furthermore, we have $(1-\zeta)^2 \nmid A_n/3$
 because if $(1-\zeta)^2 |A_n/3$ in $\Z [\zeta]$, then we have $3 |A_n/3 =n/3 +1+\zeta$. However, 
 this does not hold because $\{1, \zeta \}$ is an integral basis of $\Z [\zeta]$.
 We conclude that $
 (1-\zeta) || A_n/3$.
 On the other hand, we have $(1-\zeta)^2 ||ec$.
 According to the aforementioned facts, we have
 \begin{equation}\label{eq:ideal-3part-1}
 (1-\zeta) || (ec, A_n/3).
 \end{equation}
 Then, we consider prime elements other than $1-\zeta$.
 Since any prime number $p$ dividing $ec$ satisfies $p=3$ or $p\equiv 1 \pmod{3}$ (\cite[Lemma~4.1]{HA1})
 and the only prime ideal that  containes both $A_n$ and $A_n'$ is $(1-\zeta)$, we have
 \begin{equation}\label{eq:ec-decomp1}
 \frac{ec}{3}=(\pi_1 \cdots \pi_k)(\pi_1' \cdots \pi_k'),\quad \pi_1 \ldots \pi_k |A_n, \quad \pi_1' \ldots \pi_k' |A_n'
 \end{equation}
 where $\pi_1,\ldots,\pi_k$ are prime elements dividing $A_n$, and $\pi_1',\ldots, \pi_k'$ are their conjugates, respectively.
 From (\ref{eq:nuI-gen-1}),(\ref{eq:ideal-3part-1}) and (\ref{eq:ec-decomp1}), we obtain the following:
 \begin{align}
 \nu (I)= (\nu (g)) & =9 (1-\zeta)(\pi_1\cdots \pi_k) ((1-\zeta)(\pi_1' \cdots \pi_k'),  (1-\zeta)^{-1} (\pi_1\cdots \pi_k)^{-1} A_n/3) \notag\\
 & =9 (1-\zeta) (\pi_1\cdots \pi_k )(1) \label{eq:nu-gamma1} \\
 & =(\delta),\notag
 \end{align}
 where $\delta :=9(1-\zeta)(\pi_1 \cdots \pi_k)$. 
Since the integers $a_0$ and $a_1$ satisfy $(a_0+a_1\zeta)(a_0+a_1\zeta^2)=a_0^2-a_0a_1+a_1^2=ec$ and $a_0+a_1 \zeta | A_n$,
 we have $(\delta)=9(a_0+a_1\zeta)$.
Put $x:=9(a_0+a_1\sigma )\in \Z [G]$. From (\ref{eq:nu-gamma1}), we obtain $\nu (g)=9v(a_0+a_1\zeta) $ for some
 $v\in \Z [\zeta]^{\times}$. Since $\Z [\zeta]^{\times} =\{\pm 1,\pm\zeta,\pm\zeta^2 \}$ and $\Z[G]^{\times}=
 \{ \pm 1, \pm \sigma, \pm \sigma^2 \}$, there exists $\xi \in \Z [G]^{\times}$ that satisfies $\nu (\xi)=v^{-1}$, and we have
 $\nu (\xi g)=9(a_0+a_1\zeta)=\nu (x)$. We conclude that $\xi g-x \in \mathrm{Ker}(\nu)=(1+\sigma+\sigma^2)$.
 Therefore, there exists $m\in \Z$ that satisfies
 \[
 (\xi g-x)\cdot  \frac{\rho}{\ell} =m \mathrm{Tr}_{L_n/\Q}\left( \frac{\rho}{\ell} \right).
 \]
 By acting $e_{9de}$ on this equation, we obtain $ e_{9de} (\xi g-x)\cdot (\rho/\ell )=0$.
 Since $\alpha=g\cdot  (\rho/\ell)$ is an element of $e_{9de}\cdot  \calO_{L_n}$, we have
 $(e_{9de} g)\cdot (\rho/\ell)=g \cdot (\rho/\ell)$. Combining all these,  we have the following:
 \begin{align*}
\xi\cdot \alpha= (\xi g) \cdot \frac{\rho}{\ell} & =(e_{9de} \xi g) \cdot  \frac{\rho}{\ell} \\
 & =(e_{9de}x)\cdot \frac{\rho}{\ell} \\
 & =\frac{1}{ec^2}(3a_0+3a_1\sigma -(a_0+a_1)(1+\sigma+\sigma^2))\cdot \rho \\
 & =\frac{1}{ec^2} (3a_0\rho+3a_1\rho'-n(a_0+a_1)).
 \end{align*}
 Since $\xi \in \Z [G]^{\times}$, we have $\Z [G]\cdot \alpha =\Z[G]\cdot  (\xi \cdot \alpha)$, and from (\ref{eq:O1}) and
 (\ref{eq:alpha1}),  the assertion is obtained.
 
 Next, we consider the case where $n\equiv 0,6 \pmod{9}$.
 If $n=0$, then we can see directly that the assertion of the theorem holds. In the following, we assume that
 $n\ne 0$, hence $d(\rho, \rho',\rho'')\ne 0$ and $\{ \rho, \rho',\rho'' \}$ is a normal basis of $L_n$.
 In this case, we have $3\nmid d,\, 3 ||e,\, 3 \nmid c$ and the conductor is 
$\mathfrak f_{L_n}=3de$. The corresponding group of Dirichlet characters is $\mathscr{X}=\Phi_{3de} \coprod \Phi_{de}$ with
$\Phi_{3de}=\{ \chi, \chi^2 \},\, \Phi_{de}= \{\bf 1 \}$,
where $\bf 1$ is the trivial character and $\chi : G \to \C^{\times},\
\sigma \mapsto \zeta_3$. 
We have $e_{3de}=e_{\chi}+e_{\chi^2}= (2-\sigma-\sigma^2)/3,\, 
e_{de}=e_{\bf 1} =(1+ \sigma+\sigma^2)/3$, $\calD(\mathfrak f_{L_n})=\{ 3de,de \}$,
$\calA_{L_n/\Q}=\Z [G][e_{3de},e_{de}]$ and
$\calO_{L_n}=\Z [G] \eta_{3de} \oplus \Z [G] \eta_{de}$.
Similarly to the case of $n\equiv 3 \pmod{9}, \ n\not\equiv 12 \pmod{27}$,
we have the following:
\begin{equation}\label{eq:O2}
\calO_{L_n} =\Z [G] \cdot \eta_{3de} \oplus \Z [G] \cdot \eta_{de} = \Z [G]\cdot  \eta_{3de} \oplus \Z   =\calA_{L_n/\Q} \cdot (\eta_{3de}+1).
\end{equation}
Let $\alpha$ be an element of $\calO_{L_n}$ that satisfies the following:
\begin{equation}\label{eq:alpha2}
\Z [G]\cdot \alpha=\Z [G]\cdot \eta_{3de}.
\end{equation}
 We have
$e_{3de} \cdot \calO_{L_n} =\Z [G]\cdot \alpha$
from (\ref{eq:em}), (\ref{eq:O2}) and (\ref{eq:alpha2}). Let $ \{1,\phi,\psi \}$ be the integral basis of $L_n$ given by
Lemma~\ref{lem:IB}. 
Let $\ell :=3ec^2$. 
Similarly to the case of $n\equiv 3 \pmod{9}, \ n\not\equiv 12 \pmod{27}$,
there exists $g\in \Z [G]$
satisfying
\begin{equation}\notag
\alpha= g\cdot \frac{\rho}{\ell}
\end{equation}
and
we have 
\begin{equation}\notag
e_{3de}\cdot  \calO_{L_n} =\Z [G]\cdot \alpha=(e_{3de}\cdot  \phi) \Z +(e_{3de}\cdot \psi)\Z,
\end{equation}
where
\begin{align*}
e_{3de}\cdot \phi =g_1. \frac{\rho}{\ell}, \hspace{10mm} & g_1:= \frac{\ell}{3c} (2-\sigma -\sigma^2), \\
e_{3de}\cdot \psi =g_2. \frac{\rho}{\ell}, \hspace{10mm} & g_2:=\frac{\ell}{3ec^2} ((2n+3)+(3-n)\sigma -(n+6)\sigma^2) \notag .
\end{align*}
By calculating  the image of the ideal $I:=(g)=(g_1,g_2)$ by the map
$\nu : \Z [G] \to \Z [\zeta]$,
we obtain the ideal of $\Z [\zeta]$:
\begin{equation}\notag
\nu (I) =( \nu (g)) =(\nu (g_1), \nu(g_2)).
\end{equation}
where 
 \begin{align*}
 \nu (g_1) & =\frac{\ell}{3c} (2-\zeta-\zeta^2)=3ec, \\
 \nu (g_2) & = \frac{\ell}{3ec^2} ((2n+3)+(3-n)\zeta -(n+6)\zeta^2) =9 \times \frac{A_n}{3} 
 \end{align*}
 with $A_n/3=n/3+\zeta+1 \in \Z [\zeta]$. 
 Therefore, we have 
 \begin{equation}\label{eq:nuI-gen-2}
 \nu (I)=(\nu (g))=9 (ec/3,A_n/3).
 \end{equation}
 We consider the decompositions of $ec/3$ and $A_n/3$ into  products of prime elements. 
 Since $3 \nmid ec/3$, we have
 \begin{equation}\label{eq:ideal-3part-2}
 (1-\zeta) \nmid  (ec/3, A_n/3).
 \end{equation}
 For the  other prime elements, let
 \begin{equation}\label{eq:ec-decomp2}
 \frac{ec}{3}=(\pi_1 \cdots \pi_k)(\pi_1' \cdots \pi_k'),\quad \pi_1 \ldots \pi_k |A_n, \quad \pi_1' \ldots \pi_k' |A_n'
 \end{equation}
 where $\pi_1,\ldots,\pi_k$ are prime elements dividing $A_n$, and $\pi_1',\ldots, \pi_k'$ are their conjugates, respectively.
 Then, from (\ref{eq:nuI-gen-2}),(\ref{eq:ideal-3part-2}) and (\ref{eq:ec-decomp2}), we obtain the following:
 \begin{align}
 \nu (I)= (\nu (g)) & =9 (\pi_1\cdots \pi_k) ((\pi_1' \cdots \pi_k'),  (\pi_1\cdots \pi_k)^{-1} A_n/3) \notag\\
 & =9 (\pi_1\cdots \pi_k )(1) \notag \\
 & =(\delta),\notag
 \end{align}
 where $\delta :=9(\pi_1 \cdots \pi_k)$. 
Since the integers $a_0$ and $a_1$ satisfy $3(a_0+a_1\zeta)(a_0+a_1\zeta^2)=3(a_0^2-a_0a_1+a_1^2)=ec $ and
$ a_0+a_1\zeta |A_n$, 
 similarly to the case of $n\equiv 3 \pmod{9}, \ n\not\equiv 12 \pmod{27}$,
we can show that 
 \begin{align*}
\xi\cdot \alpha= (\xi g)\cdot  \frac{\rho}{\ell} & 
  =\frac{1}{ec^2}(3a_0+3a_1\sigma -(a_0+a_1)(1+\sigma+\sigma^2))\cdot \rho \\
 & =\frac{1}{ec^2} (3a_0\rho+3a_1\rho'-n(a_0+a_1)),
 \end{align*}
 for some $\xi \in \Z [G]^{\times}$. From (\ref{eq:O2}) and
 (\ref{eq:alpha2}),  the assertion is obtained.
\end{proof}

\section{Examples}\label{sec:ex}
\begin{ex}\label{ex:1}
We give examples for the case $n\equiv 3\pmod{9}, n\not\equiv 12 \pmod{27}$. A pair of integer $\{a_0, a_1 \} $ 
explicitly given as follows satisfies the conditions in Theorem~\ref{theo:OA}: $ec=a_0^2-a_0a_1+a_1^2$ and $a_0+a_1 \zeta  |A_n$.
Therefore, according to Theorem~\ref{theo:OA},  $\calO_{L_n} =\Z [G]\cdot \alpha \oplus \Z=\calA_{L_n/\Q}\cdot (\alpha+1)$ holds for
 the element $\alpha$ given as follows.
\begin{enumerate}
\item[(1)] Let $n$ be an integer that satisfies $n\equiv 3 \pmod{9}, n\not\equiv 12 \pmod{27}, n\ne 237$ and $1\leq n\leq 300$
(there are $22$ integers $n$ satisfying these conditions).
We have $e=1, c=3$.
A pair of integers is $a_0=1, a_1=-1$, and hence the corresponding element $\alpha \, (\in \calO_{L_n})$ is  $\alpha=(\rho-\rho')/3$.
\item[(2)] Let $n=237$. We have $\Delta_n=3^3\cdot 7^2\cdot 43$ and $e=7,c=3$.
A pair of integers is $a_0=4, a_1=-1$, and hence the  corresponding element $\alpha\,  (\in \calO_{L_n})$ is  $\alpha=(4\rho-\rho'-237)/21$.
\end{enumerate}
\end{ex}
\begin{ex}\label{ex:2}
We give examples for the case $n\equiv 0,6\pmod{9}$. A pair of integer $\{a_0, a_1 \} $ 
explicitly given as follows satisfies the conditions in Theorem~\ref{theo:OA}: $ec=3(a_0^2-a_0a_1+a_1^2)$ and $a_0+a_1 \zeta  |A_n$.
Therefore, according to Theorem~\ref{theo:OA}, $\calO_{L_n} =\Z [G]\cdot  \alpha \oplus \Z=\calA_{L_n/\Q}\cdot (\alpha+1)$ holds for
 the  corresponding element $\alpha$ given as follows.
\begin{enumerate}
\item[(1)] Let $n$ be an integer that satisfies $n\equiv 0,6 \pmod{9}, n \ne  54,90$ and $1\leq n\leq 100$
(there are $20$ integers $n$ satisfying these conditions).
We have $e=3, c=1$.
A pair of integers is $a_0=1, a_1=0$, and hence the  corresponding element $\alpha \, (\in \calO_{L_n})$ is  $\alpha=\rho-n/3$.
\item[(2)] Let $n=54$. We have $\Delta_n=3^2\cdot 7^3$ and $e=3,c=7$.
A pair of integers is $a_0=2, a_1=-1$, and hence the  corresponding element $\alpha (\in \calO_{L_n})$ is  $\alpha=(2\rho-\rho'-18)/49$.
\item[(3)] Let $n=90$. We have $\Delta_n=3^2\cdot 7^2 \cdot 19$ and $e=3\cdot 7 ,c=1$.
A pair of integers is $a_0=3, a_1=1$, and hence the  corresponding element $\alpha \, (\in \calO_{L_n})$ is  $\alpha=(3\rho+\rho'-120)/7$.
\end{enumerate}
\end{ex}

\end{document}